\documentclass[reqno,12pt]{amsart}
\usepackage{graphicx} 
\usepackage{booktabs}
\usepackage[numbers,square]{natbib}

\makeatletter
\def\l@section{\@tocline{1}{0pt}{0pc}{}{}}
\def\l@subsection{\@tocline{2}{0pt}{2.3em}{}{}}
\makeatother

\allowdisplaybreaks

 \usepackage{amsfonts,amssymb,enumerate,epsf,amsmath,bbm,amsthm,amscd}
 \usepackage[mathscr]{eucal}
 \usepackage{xcolor}
 \usepackage{tikz}
\usetikzlibrary{arrows.meta, positioning}
 \usepackage{standalone}
\usepackage{hyperref}
\usepackage{verbatim}
\usepackage{pgfplots}
\usepackage{pgfplotstable}
\usepgfplotslibrary{colorbrewer}
\usepgfplotslibrary{groupplots}
\usetikzlibrary{patterns}
\usetikzlibrary{arrows}
\usepackage[utf8]{inputenc} 
\usepackage[T1]{fontenc}
\usepackage{microtype}
\usepackage{mathrsfs}
\usepackage{float}
\usepackage{subcaption}
\hypersetup{colorlinks=true, linkcolor=blue, citecolor=blue,  filecolor=blue, urlcolor=blue}
\usepackage{xcolor}
\definecolor{ForestGreen}{RGB}{34,139,34}

\usepackage{fullpage}

\theoremstyle{definition}
\newtheorem{thm}{Theorem}

\newtheorem{lem}{Lemma}

\definecolor{ForestGreen}{RGB}{34,139,34}

\definecolor{grigio}{RGB}{232,232,232}
\definecolor{grigio2}{RGB}{182,182,182}

\newcommand{\cal}[1]{\mathcal{#1}}

\def \cI {{\cal I}}

 \def \cX {{\cal X}}

 \def \L {{\Lambda}}

 \def \b {{\beta}}
 
 \def \D {{\Delta}}
 
\def \m {{\mu}}

 \def \h {{\eta}}
 
 \def \z {{\zeta}}

 \def \o {{\omega}}

\newtheorem{theorem}{Theorem}[section]

\newtheorem{definition}[theorem]{Definition}

\numberwithin{equation}{section}
\numberwithin{figure}{section}
\numberwithin{thm}{section}
\numberwithin{lem}{section}

\title{Optimal Strategies for Controlled Growth \\ in Metastable Kawasaki Dynamics}

\author[S. Baldassarri]{Simone Baldassarri}
\address{Gran Sasso Science Institute, Viale Francesco Crispi 7, 67100 L’Aquila, Italy}
\email{simone.baldassarri@gssi.it}
\thanks{}
\author[M. de Jongh]{Maike C. de Jongh}
\address{Department of Applied Mathematics, University of Twente, P.O. Box 217, NL-7500 AE Enschede}
\email{m.c.dejongh@utwente.nl}

\begin{document}

\begin{abstract}
In this paper, we develop a Markov decision process (MDP) formulation for the low--temperature metastable Ising model evolving according to Kawasaki dynamics in a finite box of the two--dimensional square lattice. We analyze how an external controller can guide the system to the all--occupied state by appropriately adding and moving particles at specified moments in time. To this end, we construct a reduced MDP on a constrained family of configurations having a single cluster, a regime where particle attachment is more likely than detachment. We investigate two reward structures: one that depends solely on the time to reach the target configuration, and another that incorporates action--dependent energy costs. Within this MDP framework, we characterize the exact optimal policies under both reward structures, which turn out to have a different behavior: while a purely efficiency--based criterion promotes the growth from the boundary centers of the cluster, an energy--based reward function favours the growth at the corners of the cluster.
\end{abstract}

\maketitle



\section{Introduction}

\subsection{Background and motivation}

Metastability is a hallmark of complex stochastic systems that exhibit long--lived quasi--stationary regimes before undergoing rare transitions to more stable configurations. In lattice spin systems, interacting particle systems, and diffusive dynamics, such transitions are exponentially unlikely at low temperature and are typically driven by the rare crossing of energy barriers. Traditionally, metastability has been studied through two main approaches: the pathwise approach \cite{Cassandro1984,Manzo2004,Olivieri2005} and potential theory 
\cite{Bovier2015, Bovier2002}. More recent approaches are developed in \cite{Beltrn2010,Beltrn2014,Bianchi2016,Bianchi2020,landim2021resolvent}. These classical analyses characterize the exponential asymptotics of transition rates through Eyring–Kramers–type formulas, which link escape times to the geometry of the energy landscape. See the review \cite{Jacquier2025} for more details. While these methods have produced deep insights into metastable phenomena, they mainly focus on {\it static} transition probabilities and asymptotic characterizations of rare events, rather than on {\it dynamic mechanisms} or {\it strategies} to influence or control those transitions.

In this paper, we focus on the Ising model evolving according to Kawasaki dynamics inside a finite square box, i.e., particles hop around randomly subject to hard--core repulsion and nearest--neighbour attraction. In the low--temperature regime, the relaxation from a metastable configuration typically occurs through the nucleation of a critical droplet of the stable phase. For the two--dimensional Ising lattice gas with nearest--neighbor interactions, the critical nucleus takes the form of a minimal square droplet whose size depends on the binding energy and the per--particle activation energy. Once this droplet appears, it rapidly grows to invade the system. This mechanism and the associated exponential scaling of exit times have been rigorously characterized in \cite{Bovier2005, denHollander2000}. Subsequent studies have extended these analyses to three dimensions \cite{denHollander2003}, anisotropic interactions \cite{Baldassarri2021,Baldassarri2022strong,Baldassarri2022weak,NOS2005}, infinite volume \cite{baldassarri2023droplet,baldassarri2024homogeneous, Bovier2010,Gaudilliere2009}, two--particle systems \cite{denHollander2011,denHollander2012} and hexagonal lattice \cite{baldassarri2023metastability}.

For such systems, direct numerical simulation at low temperatures becomes nearly intractable. The origin of this difficulty is well known: when the inverse temperature $\beta$ is large, the expected exit time from a metastable basin grows exponentially with $\beta$, typically as $\mathbb{E}[\tau_{exit}]\sim\exp(\beta\Gamma)$ for some energy barrier $\Gamma>0$, which corresponds to the energy of the critical droplet. As a consequence, trajectories remain trapped for extremely long times in metastable valleys before a single relevant transition occurs. Because local relaxation occurs rapidly while escapes are exponentially rare, brute--force Monte Carlo simulations spend almost all their time exploring local fluctuations, yielding little information about the transitions themselves.
Moreover, rare--event estimation may lead to large relative variance: to obtain reliable estimates of mean exit times or transition probabilities, one would need an exponentially large number of independent realizations. Common variance--reduction or coarse--graining techniques often distort the energy landscape, thereby destroying the asymptotic properties one seeks to understand. The situation is further complicated in the conservative setting of Kawasaki dynamics, where the number of particles is fixed and the configuration space is both combinatorially large and constrained. 
For these reasons, the modern study of metastability increasingly demands {\it computational} and {\it algorithmic perspectives}.

To overcome these difficulties, we adopt a complementary viewpoint by formulating the acceleration of Kawasaki dynamics within the framework of Markov Decision Processes (MDPs) \cite{Puterman}. Since its introduction in the 1950s, the MDP has been a widely applicable tool for modeling sequential decision making problems under uncertainty. The general setup consists of a system that evolves over time on a certain \textit{state space}. The dynamics of this system are influenced by an external \textit{decision maker}, who can execute \textit{actions} at each time in a discrete set of \textit{decision epochs}. The way in which the decision maker selects these actions is described by means of a \textit{policy}. After taking an action, the decision maker receives a reward determined by the current state of the system and the selected action. The quality of a policy is quantified by the so-called \textit{value function}, typically defined as the expected discounted cumulative reward. A policy is \textit{optimal} if it maximizes this value function.

In a range of contexts, the MDP has proved to be a valuable framework for systems characterized not only by temporal but also by spatial stochastic structures, for example in wildfire and forest management~\cite{Altamimi2022, Diao2020,Haksar2020,Haksar2019}, in the development of intervention strategies for networked contagion~\cite{Nasir2023,Song2023}, and in controlling the assembly of colloidal particles in crystalline structures~\cite{Tang2016,Tang2017}.
In each of these settings, the MDP framework is used to control stochastic processes that propagate over a spatial structure. However, this spatial richness comes at a cost: the state space grows exponentially with the number of sites, rendering exact optimization computationally intractable. Consequently, much of the research on these control problems focuses on approximation techniques or reduced models that preserve essential spatial features, such as locality and interaction structure, while allowing tractable computation of optimal policies. In \cite{JBL25}, the MDP framework is used to derive an optimal mechanism to control the two-dimensional Ising model evolving under Metropolis dynamics in the metastable regime. To address the large configuration space, a simplified MDP was introduced, based on a reduction of the state space to the local minima of the Hamiltonian. The present work extends this approach to the control of the low-temperature Kawasaki dynamics. The MDP perspective on stochastic systems in metastable regimes offers several key advantages, both conceptually and computationally.
\begin{itemize}
    \item {\it Algorithmic acceleration}: by interpreting the biasing of transition rates as a control problem, one can design optimal or near--optimal acceleration strategies that dramatically reduce simulation times while preserving the correct statistical weights.
    \item {\it Theoretical clarity}: the value function of an MDP can be expressed through a variational principle balancing reward maximization and the probabilistic cost of rare transitions. This formulation closely parallels the structure of large deviation rate functions, thereby unveiling a direct link between optimal control and metastable dynamics.
    \item {\it Spatial adaptivity}: unlike global temperature rescaling, which alters all dynamics equally, MDP--based controls target specific regions of configuration space, thereby accelerating escape near bottlenecks while maintaining the natural motion within metastable wells.
    \item {\it Compatibility with data--driven methods}: in the absence of explicit coarse-grained models, reinforcement learning can infer optimal MDP policies directly from simulation data, enabling adaptive control and efficient exploration of metastable landscapes.
\end{itemize}
In summary, the MDP framework offers a dynamic, optimization--based lens on metastability, complementing the traditional static and asymptotic approaches. It provides both (i) a theoretical formalism to describe low--temperature behavior in terms of optimal control problems and (ii) a practical basis for the design of accelerated numerical schemes. In the low--temperature limit, the optimal MDP strategies recover the leading exponential scaling of the metastable transition times, while offering improved computational tractability.

\subsection{Main contributions}
\label{sec:maincontributions}
In this paper, we propose an MDP framework for analyzing low--temperature metastability for Kawasaki dynamics in a finite box inside the two--dimensional square lattice. The state space consists of particle configurations; possible actions correspond to local particle--exchange moves; and transitions are governed by Kawasaki–Metropolis rates (see \eqref{defkaw} below for more details). The objective (or reward structure) can be engineered to represent hitting a target configuration (e.g. nucleation) while minimizing expected cost, which naturally leads to an optimal policy that effectively brings the system through rare--event barriers. Through this lens, metastable escape becomes not only a statistical rarity but also a controlled dynamical phenomenon amenable to policy--based optimization. We consider two different reward structures: one that focuses exclusively on reaching the target configuration in the most efficient manner, and one that introduces action-dependent costs, mimicking the energy costs of such transitions in the uncontrolled system. Our main contributions are twofold. First, we provide a detailed formulation of the problem of controlling the low-temperature Kawasaki dynamics within the framework of an MDP. Second, we introduce a simplified MDP defined over a specific class of configurations that contain a single particle cluster, in which it is more likely for new particles to enter the box than for particles to detach from a cluster. Finally, we derive the structure of the exact optimal policies in this simplified MDP under each of the two reward structures. Our results demonstrate that incorporating an energy cost into the reward structure yields a markedly different optimal strategy. When the objective is simply to find the most efficient path to the full box, the optimal policy favors growth from the center of the cluster boundary, whereas an energy-based reward function instead induces growth at the cluster corners.

\subsection{Outline}
The structure of the paper is as follows. In Section \ref{sec:model}, we introduce the Ising MDP for Kawasaki dynamics and its auxiliary counterpart, for which we state our main results. The corresponding proofs appear in Section \ref{sec:proof}. Section \ref{sec:concl} presents our conclusions and outlines directions for future work. 

\section{Model and results}
\label{sec:model}

\subsection{The Markov decision process}
\label{sec:markovproc}

In this section, we detail the formal MDP construction, which is defined as a tuple $(S,A,P,r)$, where each element is specified as follows~\cite{Puterman}.

\begin{itemize}
    \item The state space $S$ is the set of all possible states that the system can occupy.
    \item Given a state $s \in S$, the action space $A(s)$ consists of all actions available to the decision maker in that state. Let $A := \cup_{s \in S} A(s)$. 
    \item The transition probability kernel $P: S \times A \times S$ governs the evolution of the MDP. We write $P(s'|s, a)$ for the probability of the system transitioning from state $s$ to state $s'$ when action $a$ was taken. 
    \item The reward function $r: S \times A \rightarrow \mathbb{R}$ assigns the immediate reward $r(s, a)$ received after choosing action $a \in A(s)$ from state $s \in S$. 
\end{itemize}
The times at which the decision maker may take an action are collected in a set $T$ of \textit{decision epochs}. These actions are selected according to a policy that specifies which action to take at each decision epoch. We consider only policies that are \textit{stationary} and \textit{deterministic}. Such a policy is characterized by a single deterministic \textit{decision rule} $d: S \rightarrow A$, prescribing an action $a \in A(s)$ for each state $s \in S$. Let the set of all stationary deterministic policies be denoted by $\Pi$.

The objective of the decision maker is to design a policy that generates the best possible sequence of rewards. Since the rewards are random variables, finding the best policy requires an optimality criterion in terms of random sequences. We assess the quality of a policy $\pi \in \Pi$ by means of its \textit{expected total discounted reward}, or its \textit{value function} $v^{\pi}_{\lambda}: S \rightarrow \mathbb{R}$, given by
\begin{equation}
    v^{\pi}_{\lambda}(s) = \mathbb{E}^{\pi}_s\left[\sum\limits_{t=1}^{\infty}\lambda^{t-1}r(X^{\pi}_t, Y^{\pi}_t)\right],
\end{equation}
where $X^{\pi}_t$ denotes the state of the system at decision epoch $t$, $Y^{\pi}_t$ the selected action at decision epoch $t$, $\lambda \in (0,1)$ is a so-called \textit{discount factor}, and $s \in S$ is the initial state. 

A policy $\pi^* \in \Pi$ is called \textit{optimal} if it satisfies
\begin{equation}
    v^{\pi^*}_\lambda(s) \geq v^{\pi}_{\lambda}(s), \quad s \in S,
\end{equation}
for all $\pi \in \Pi$. Simplifying notation, we will denote the value function of an optimal policy $\pi^*$ by $v^*_{\lambda}(s)$, $s \in S$. 

The optimal values and policies in infinite horizon models can be characterized by the so-called \textit{optimality equations} or \textit{Bellman equations}:
\begin{equation}
    v_{\lambda}(s) = \max_{a \in A(s)}\{r(s, a) + \sum\limits_{j \in S} \lambda P(j|s, a)v_{\lambda}(j)\}.
\end{equation}
The following theorem establishes the usefulness of these equations in identifying optimal policies.
\begin{thm}\label{optimality_thm}\cite{Puterman}
A policy $\pi^* \in \Pi$ is optimal if and only if $v^{\pi^*}_{\lambda}$ is a solution to the optimality equations.
\end{thm}

\subsection{Definition of the Ising MDP for Kawasaki dynamics}
Let $\Lambda=\{0,1,...,L\}^2\subset\mathbb{Z}^2$ be a finite box centered at the origin. The side length $L$ is fixed, but arbitrarily large. Let
\begin{equation}\label{inbd}
\partial^- \Lambda:= \{x\in\L: \exists\; y \notin\Lambda\: |y-x|=1\},
\end{equation}
be the interior boundary of $\Lambda$ and let $\Lambda_0:= \Lambda\setminus\partial^- \Lambda$ be the interior of $\Lambda$.

With each $x\in\Lambda$ we associate an occupation variable $\eta(x)$, assuming values 0 (empty site) or 1 (occupied site). A lattice configuration is denoted by $\eta\in\cX=\{0,1\}^\Lambda$. Each configuration $\h\in \cX$ has an energy given by the following Hamiltonian:
\begin{equation}\label{hamilt} 
H(\eta)= -U \sum_{\{x,y\}\in   \Lambda_0^{*}} \eta(x)\eta(y) + \D \sum _{x\in \L} \eta (x), 
\end{equation}
where $ \Lambda_0^{*}$ denotes the set of {\it unoriented} bonds between nearest--neighbour sites in $\Lambda_0$, i.e., there is a {\it binding energy} $-U<0$ between neighbouring particles, and $\Delta>0$ is the per--particle {\it activation energy}. Taking $\Delta\in(U,2U)$ corresponds to the metastable regime, and we further assume that $\Delta>\frac{3}{2}U$ to avoid trivial cases, i.e., the side length of the critical droplet is larger than 2. We refer the interested reader to \cite{denHollander2000} for more details. The grand--canonical Gibbs measure associated with $H$ is
\begin{equation}\label{misura} 
\m(\eta)= {  e^{- \b H(\eta) }\over Z_\beta}, \qquad \h\in
\cX, 
\end{equation}
where $\beta>0$ is the inverse temperature, and
\begin{equation}\label{partfunc} 
Z_\beta=\sum_{\eta\in {\cal X}}e^{-\b H(\eta)}
\end{equation}
is the so--called {\it partition function}.

We define now Kawasaki dynamics on $\Lambda$ with boundary conditions that mimic the effect of an infinite gas reservoir outside $\Lambda$ with density $ \rho = \exp(-\Delta\beta)$. Let $b=(x \to y)$ be an oriented bond, i.e., an {\it ordered} pair of nearest neighbour sites, and define

\begin{equation}\label{Loutindef}
\begin{array}{lll}
\partial^* \Lambda^{out} &:=& \{b=(x \to y): x\in\partial^- \Lambda,
y\not\in\Lambda\},\\
\partial^* \Lambda^{in}  &:=& \{b=(x \to y): x\not\in
\Lambda, y\in\partial^-\Lambda\},\\
\Lambda^{*, orie} &:=& \{b=(x \to y): x,y\in\Lambda\},
\end{array}
\end{equation}
and put $ \bar\Lambda^{*, orie}:=\partial^* \Lambda ^{out}\cup \partial^* \Lambda ^{in}\cup\Lambda^{*,\;orie}$. Two configurations $  \eta,
\eta'\in {\cal X}$ with $ \eta\ne \eta'$, are said to be {\it communicating states} if there exists a bond $b\in  \bar\L^{*,orie}$ such that $ \eta' = T_b \eta$, where $T_b \eta$ is the configuration obtained from $ \eta$ in any of these ways:
\begin{itemize}
	\item
	for $b=(x \to y)\in\Lambda^{*,\;orie}$, $T_b \eta$ denotes the configuration obtained from $\eta$ by interchanging particles along $b$:
	\begin{equation}\label{Tint}
	T_b \h(z) =
	\left\{\begin{array}{ll}
	\h(z) &\mbox{if } z \ne x,y,\\
	\h(x) &\mbox{if } z = y,\\
	\h(y) &\mbox{if } z = x.
	\end{array}
	\right.
	\end{equation}
	
	\item
	For  $b=(x \to y)\in\partial^*\Lambda^{out}$ we set:
	\begin{equation}\label{Texit}
	T_b \h(z) =
	\left\{\begin{array}{ll}
	\h(z) &\mbox{if } z \ne x,\\
	0     &\mbox{if } z = x.
	\end{array}
	\right.
	\end{equation}
	This describes the annihilation of a particle along the border;
	
	\item
	for  $b=(x\to y)\in\partial^*\Lambda^{in}$ we set:
	\begin{equation}\label{Tenter}
	T_b \h(z) =
	\left\{\begin{array}{ll}
	\h(z) &\mbox{if } z \ne y,\\
	1     &\mbox{if } z=y.
	\end{array}
	\right.
	\end{equation}
	This describes the creation of a particle along the border.
	
\end{itemize}

\noindent
The Kawasaki dynamics is  the discrete time Markov chain
$(\eta_t)_{t\in \mathbb{N}}$ on state space $ {\cal X} $ given by
the following Metropolis transition probabilities: for  $  \eta\not= \eta'$:
\begin{equation}\label{defkaw}
P_\beta( \eta,  \eta'):=\left\{
\begin{array}{ll}
{ |\bar\Lambda^{*,orie}|}^{-1} e^{-\b[H( \eta') - H( \eta)]_+}
&\mbox{if }  \exists b\in \bar\Lambda ^{*,orie}: \eta' =T_b \eta,   \\
0   &\mbox{ otherwise, }  
\end{array} 
\right.
\end{equation}
where $[a]_+ =\max\{a,0\}$ and $P_\beta(\h,\h):=1-\sum_{\h'\neq\h}P_\beta(\h,\h')$.

In what follows we introduce the notions of {\it downhill configurations} and {\it susceptible bonds}. A configuration $\eta' \in \mathcal{X}$ is called downhill configuration of a configuration $\eta \in \mathcal{X}$ if there exists a sequence $\omega = (\eta_1, \eta_2, ..., \eta_{\ell})$ for some $\ell \in \mathbb{N}$ that satisfies $\eta_1 = \eta$, $\eta_{\ell} = \eta'$, $H(\eta_{t+1}) \leq H(\eta_t)$ for each $t = 1, ..., \ell -1$ and $\eta_{t+1} = T_b \eta_t$ for some $b \in \bar\Lambda^{*, orie}$ for each $t = 1, ..., \ell -1$. We say that a bond $b \in \bar\Lambda^{*, orie}$ is susceptible in a configuration $\eta \in \mathcal{X}$ if $H(T_b \eta) \leq H(\eta)$. Note that this is equivalent to the condition $\lim_{\beta \rightarrow \infty} P_{\beta}(\eta, T_b \eta) > 0$.

Each decision epoch, the decision maker can select a bond $b=(x \to y)\in\Lambda^{*,\;orie}$ satisfying $\eta(x) \neq \eta(y)$ and interchange particles along this bond. The aim is to reach a certain target configuration $\eta^*$, for example the full box, from a starting configuration $\eta_0$. In the low-temperature regime, the process is prone to becoming trapped in local minima of the energy function when given sufficient time to relax after an action. Therefore, motivated by \citet{JBL25}, we define an MDP ranging over local minima. In particular, the state space consists of those local minima that contain one cluster of particles and for which the entering of a new particle inside the box is more likely than the detachment of a particle from the cluster. Note that this is only true for configurations for which detaching a particle requires an energy cost of at least $2U$. We call such configurations \textit{robust}. To formally define the robust configurations, we introduce the notions of {\it communication height} and {\it stability level}.  The  communication height between a pair $\h,\h'\in\cX$ is
\begin{equation}\label{comh}
\Phi(\h,\h')= \min_{\o\colon\;\h\to\h'} \max_{\z\in\o} H(\z).
\end{equation}
We call the stability level of a state $\z \in \cX$ the energy barrier
\begin{equation}\label{stab}
V_{\z} = \Phi(\z,\cI_{\z}) - H(\z),
\end{equation}
\noindent
where $\cI_{\z}$ is the set of states with energy below $H(\z)$:
\begin{equation}\label{iz} 
\cI_{\z}=\{\eta \in \cX:\; H(\eta)<H(\z)\}. 
\end{equation}
We set $V_\z:=\infty$ if $\cI_\z$ is empty. 

Finally, if $\eta$ is a configuration with a single cluster, then we denote by $\hbox{CR}(\h)$ the smallest rectangle circumscribing that cluster.
\newpage 

\begin{definition}
    A configuration $\eta \in \mathcal{X}$ is called \textit{robust} if
    \begin{itemize}
    \item[(i)] $\eta$ is composed by a unique cluster;
    \item[(ii)] either $V_\eta > 2U$, or $V_\eta = 2U$ and the minimal side length of $\hbox{CR}(\h)$ is 2.
    \end{itemize}
\end{definition}
\begin{lem}\label{robust_rectangle}
The set of robust configurations coincides with the set of configurations having one rectangular cluster only with minimal side length larger than $1$.
\end{lem}

\begin{figure}
	\centering
	\includegraphics[width=0.8\textwidth]{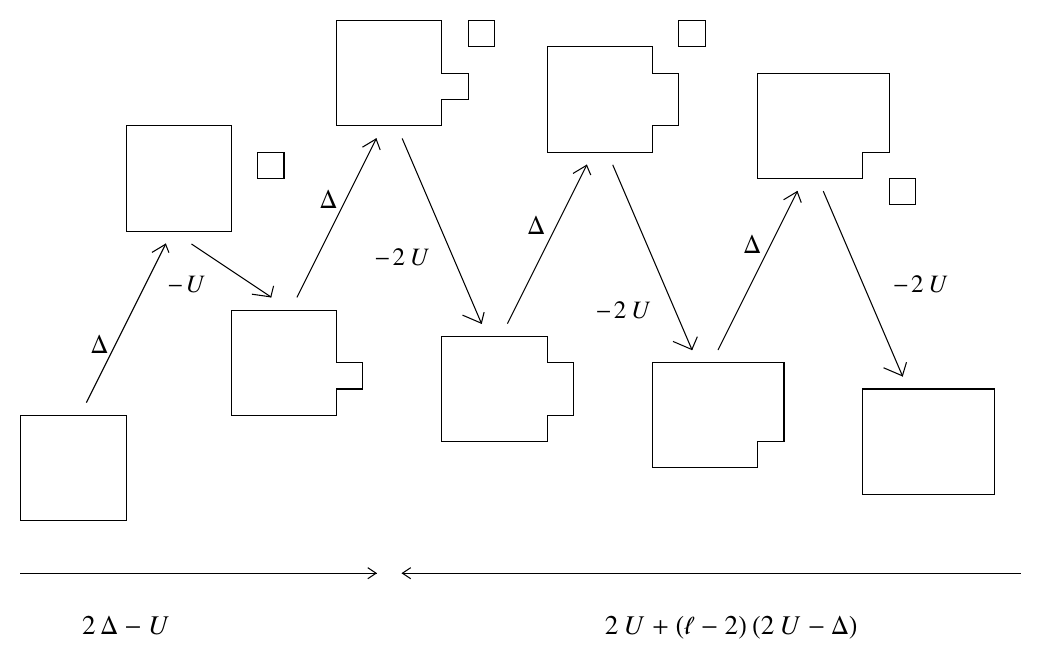}
	\caption{\small Cost of adding or removing a row of length $\ell$. This figure is taken from \cite{denHollander2000}. \normalsize}
	\label{fig:resistenza}
\end{figure}

\begin{proof}
First, we show that any configuration $\eta$ having only one rectangular cluster with minimal side length larger than $1$ is a robust configuration. We start by considering $\eta$ such that it contains a rectangular cluster with minimal side length 2. Thus, it is easy to prove that $V_\eta = 2U$. Indeed, by detaching a particle at cost $2U$, then it is possible to let it exit the box at cost $-\Delta$, and finally remove a particle at cost $U$, giving rise to a maximal energy difference equal to $2U$, because $U-\Delta<0$. Finally, by letting the last detached particle exit the box, the final energy is $H(\eta)+2U-\Delta+U-\Delta<H(\eta)$, which proves the claim.

Consider now the case in which $\eta$ contains a rectangular cluster with minimal side length $\ell$ larger than 2. We need to show that $V_\eta>2U$. We consider separately the cases $\ell<\ell_c$ and $\ell\geq\ell_c$, where $\ell_c$, which depends on the parameters $U$ and $\Delta$, is the critical length for the local model \cite{denHollander2000}. 

\smallskip
\noindent
{\bf Case $\ell<\ell_c$.} The claim follows after arguing as in the case above, where now the second particle is detached at cost $2U$ as well. The path then proceeds as in Figure \ref{fig:resistenza}.

\smallskip
\noindent
{\bf Case $\ell\geq\ell_c$.}
Since there is no allowed move at zero cost, the move with lowest energy cost is to let a particle enter the box at cost $\Delta$, since detaching any of the particles belonging to the cluster costs at least $2U$. From now on, in order to lower the energy the unique possibility is to move the free particle towards the cluster and attach it as a $1$--protuberance at cost $-U$. The total energy difference is then $\Delta-U>0$. Apart from detaching the protuberance again, or move it at zero cost along the side it is attached to, the only way to lower the energy is to let a new particle inside the box and then attach it next to the $1$--protuberance, so that the energy can decrease by $2U$. But when the new particle enters the box, the energy increases by $\Delta$, so that the total energy cost is $2\Delta-U>2U$. The path then follows as in Figure \ref{fig:resistenza}. This proves that $V_\eta>2U$. 

Finally, we consider a robust configuration $\eta$ and we want to show that it is composed by a unique rectangular cluster with minimal side length larger than 1. If $V_\eta = 2U$ and the minimal side length of $\hbox{CR}(\eta)$ is 2, then it is not geometrically possible that $\eta$ contains holes (i.e., one of the cluster of $\eta$ is not simply connected). Thus, by \cite[Proposition 5.11(iii)]{dHOS2000} it follows that $\eta$ contains a rectangular cluster with minimal side length 2. If $V_\eta>2U$, then by \cite[Proposition 16]{NOS2005} when $U_1=U_2=U$, it follows that $\eta$ is composed by a rectangular cluster with minimal side length larger than 2. This concludes the proof. 
\end{proof}

We denote the set of robust configurations by $W$. 
Let the set of decision epochs be discrete and infinite ($T = \{1, 2, \cdots\}$). At each decision epoch, given a configuration $\eta \in W$, the decision maker may choose any bond $b=(x \to y)\in\Lambda^{*,\;orie}$ satisfying $\eta(x) \neq \eta(y)$ and interchange particles along this bond. Accordingly, for each $\eta \in W$, we define the action space 
\begin{equation*}
    A(\eta) = \begin{cases}
        \{0\}, &\text{if } \eta(x) = \eta(y) \text{ for all } x,y \in \Lambda, \\
        \{b = (x \to y) \in \Lambda^{*,orie}| \eta(x) \neq \eta(y)\}, &\text{otherwise},
        \end{cases}
\end{equation*}
where the action $a = 0$ represents not interchanging any particles. 
Given an action $a \in A(\eta)$, we let $\eta^a$ denote the \textit{post-decision configuration}, i.e., the configuration obtained immediately after applying $a$ to $\eta$. 

We now describe the transition from one robust configuration to the next. Starting from~$\eta^a$, the system evolves according to Kawasaki dynamics until the first particle interchange occurs. At low temperature, such an interchange can occur only along a susceptible bond, and each susceptible bond is equally likely to be selected.
After this interchange, the resulting configuration is mapped to a new robust configuration according to the following rules. Particles that become detached from the cluster are removed from the box. New particles may be inserted only if their attachment to the cluster reduces the energy by $2U$. Furthermore, if a particle can move at zero cost along the boundary of a cluster towards a corner, we first allow the entire bar to which this particle belongs to slide around that corner, if possible, before inserting new particles. A more explicit description of these dynamics is given in Section \ref{The auxiliary MDP}. 

We consider the following two reward functions:
\begin{align}
    r_1(\eta, a) &= \begin{cases}
        1, &\text{if } \eta = \eta^*, \\
        0, &\text{otherwise,} \label{eq:rew1}
    \end{cases} \quad \eta \in W, \quad a \in A(\eta),\\
    r_2(\eta, a) &= -[H(\eta^a) - H(\eta)], \quad \eta \in W, \quad a \in A(\eta).  \label{eq:rew2}
\end{align}
The function $r_1$ reflects the objective of reaching the full box as quickly as possible. Under this reward function, the value function is directly related to the first hitting time to this target configuration. This relation is given explicitly in \citep[Corollary 2.3]{JBL25}. More details on the interpretation of the value function are provided in \cite{JSC25}. In this work, it is shown that in the regime where $\lambda$ is close to 1, maximizing the value function is equivalent to minimizing the first hitting time to the target \citep[Theorem 1]{JSC25}. Furthermore, when $\lambda$ is close to 0, only trajectories that reach the target within a small number of epochs make a significant contribution to the value function \citep[Theorem 2]{JSC25}. Hence, policies that can reach the target configuration via short paths are favoured.  

The second reward function represents the energy cost of interchanging particles along the chosen bond. Note that from a robust configuration, any feasible interchange will increase the energy of the configuration. Hence, we continue to receive negative reward until the full box is reached. This implies that the function $r_2$ reflects the objective of reaching the full box at the lowest total discounted energy cost.

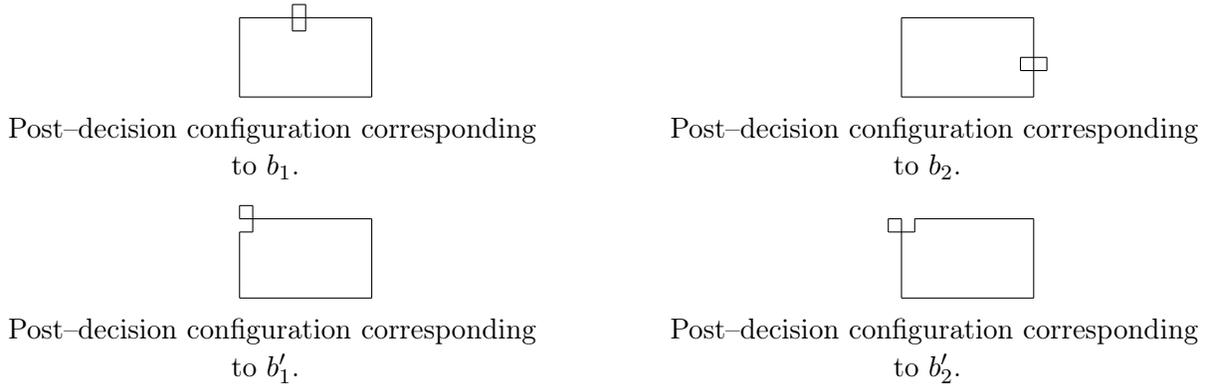
\begin{figure}
    \centering
    \captionsetup[subfigure]{justification=centering, labelformat=empty, singlelinecheck=false, width=1.2\linewidth}

    \begin{subfigure}[b]{0.45\textwidth}  
        \centering
        \begin{picture}(400,40)(-70,40)
		\thinlines
		\put(20,50){\line(1,0){50}}
		\put(20,80){\line(1,0){50}}
		\put(40,75){\line(0,1){10}}
		\put(40,85){\line(1,0){5}}
		\put(45,85){\line(0,-1){10}}
            \put(40,75){\line(1,0){5}}
		\put(70,50){\line(0,1){30}}
		\put(20,50){\line(0,1){30}}
		
        \end{picture}
        \vskip -0.3 cm
        {{\small Post--decision configuration corresponding to $b_1$.}}    
        \label{post_dec_b_1}
    \end{subfigure}
    \hfill
    \begin{subfigure}[b]{0.45\textwidth}   
        \centering 
        \begin{picture}(400,40)(-70,40)
		\thinlines
		\put(20,50){\line(1,0){50}}
		\put(20,80){\line(1,0){50}}
            \put(65,60){\line(1,0){10}}
            \put(75,60){\line(0,1){5}}
            \put(65,60){\line(0,1){5}}
            \put(75,65){\line(-1,0){10}}
		\put(20,50){\line(0,1){30}}
		\put(70,50){\line(0,1){30}}
            
		
        \end{picture}
        \vskip -0.3 cm
        {{\small Post--decision configuration corresponding to $b_2$.}}    
        \label{post_dec_b_2}
    \end{subfigure}
    
    \vskip\baselineskip

    \begin{subfigure}[b]{0.45\textwidth}
    \centering
       \begin{picture}(400,40)(-70,40)
		\thinlines
		\put(20,50){\line(1,0){50}}
		\put(20,80){\line(1,0){50}}
            \put(20,80){\line(0,1){5}}
            \put(25,80){\line(0,1){5}}
            \put(20,85){\line(1,0){5}}
		\put(20,50){\line(0,1){25}}
            \put(20,75){\line(1,0){5}}
            \put(25,75){\line(0,1){5}}
		\put(70,50){\line(0,1){30}}
		
        \end{picture}
        \vskip -0.3 cm
        {{\small Post--decision configuration corresponding to $b_1'$.}}    
        \label{post_dec_b_1'}
    \end{subfigure}
    \hfill
    \begin{subfigure}[b]{0.45\textwidth}
    \centering
       \begin{picture}(400,40)(-70,40)
		\thinlines
		\put(20,50){\line(1,0){50}}
		\put(25,80){\line(1,0){45}}
            
		\put(20,50){\line(0,1){30}}
            \put(15,75){\line(1,0){10}}
            \put(15,75){\line(0,1){5}}
            \put(15,80){\line(1,0){5}}
            \put(25,75){\line(0,1){5}}
		\put(70,50){\line(0,1){30}}
		
        \end{picture}
        \vskip -0.3 cm
        {{\small Post--decision configuration corresponding to $b_2'$.}}    
        \label{post_dec_b_2'}
    \end{subfigure}
    
    \caption{Post--decision configurations corresponding to the actions $b_1$, $b_2$, $b_1'$ and $b_2'$.}
    \label{post_dec_configs}
\end{figure}

\subsection{The auxiliary MDP}
\label{The auxiliary MDP}

Motivated by the geometrical characterization of the robust configurations given in Lemma \ref{robust_rectangle} above, we construct an auxiliary MDP, denoted by $(\hat{S}, \hat{A}, \hat{P}, \hat{r})$. To simplify the analysis, we consider a slightly modified version of the dynamics. Instead of Kawasaki dynamics with open boundary conditions, we consider its corresponding version with periodic boundary conditions, in which particles can be directly created or annihilated at sites in $\Lambda_0$ rather than at the boundary of $\Lambda$. This prevents the value function of the MDP to depend on the distance between the cluster and the boundary of $\Lambda$, which would otherwise introduce additional technicalities. In what follows, we justify why, for our purposes, the two models share the same features. During a metastable transition, both models reach the same type of critical configurations. This is because the two dynamics are local, and therefore the energy barriers associated with nucleation do not depend on whether particles must diffuse from the boundary or may be created directly in the bulk. Consequently, the dominant Arrhenius factor $\exp(\beta\Gamma^*)$ for some $\Gamma^*>0$ governing the escape rate from the metastable state is identical in the two models. The difference between the two dynamics is purely combinatorial/entropic: the bulk--injection model increases the number of distinct space--time windows that can realize a nucleation event (because injections may occur directly at sites that are favorable for nucleation instead of first having to diffuse in from the boundary). For fixed $L$ this affects only the multiplicative prefactor $K$ in the usual metastable estimate $\mathbb{E}[\tau]\sim K \exp(\beta\Gamma^*)$, where $\tau$ is the first hitting time to the stable state starting from the metastable one (see \cite[Chapter 16]{Bovier2015} for more details). Thus, since the goal is to accelerate the dynamics while preserving the correct metastable energy scale, we use this auxiliary MDP acting on a box with periodic boundary conditions, which we define as follows.

Recall that a configuration is robust if and only if it has one rectangular cluster with minimal side length of 2, by Lemma \ref{robust_rectangle}. In constructing the auxiliary MDP, we identify as equivalent all robust configurations with rectangular clusters of the same size $i \times j$, $i,j = 2, \ldots, L-2, L$ and represent such an \textit{equivalence class} by a state $(i,j)$. Hence, we define the state space $\hat{S}$ as
\begin{equation*}
    \hat{S} = \{(i,j)|i,j = 2, ..., L-2,L\}.
\end{equation*}
Given a state $(i,j) \in \hat{S}$, let the associated action space be given by 
\begin{equation*}
    \hat{A}(i,j) =
    \begin{cases}
        \{0\}, &\text{if } i = j = L, \\
        \{b_1\}, &\text{if } i = L, \quad j = 2, \ldots, L-2, \\
        \{b_2\}, &\text{if } i = 2, \ldots, L-2, \quad j = L, \\
        \{b_1', b_2, b_2'\}, &\text{if } i = 2, \quad j = 3, \quad \ldots, L-2, \\
        \{b_1, b_1', b_2'\}, &\text{if } i = 3, \ldots, L-2, \quad j = 2, \\
        \{b_1', b_2'\}, &\text{if } i = j = 2, \\
        \{b_1, b_1', b_2, b_2'\}, &\text{otherwise,}
    \end{cases}
\end{equation*}
where the actions $b_1$, $b_2$ correspond to interchanging particles along a bond at the horizontal and vertical side of the rectangle respectively, but not at the corner, and the actions $b_1'$, $b_2'$ correspond to interchanging particles along a bond at the horizontal and vertical side of one of the corners of the rectangle, respectively. The action $0$ again corresponds to doing nothing. The post--decision configurations belonging to actions $b_1$, $b_2$, $b_1'$ and $b_2'$ are illustrated in Figure \ref{post_dec_configs}. Those corresponding to actions $b_1$ and $b_2$ can be defined as the configurations obtained from $(i,j)\in\hat S$ by detaching a particle, not from a corner, towards north/south and east/west, respectively, while those corresponding to actions $b_1'$ and $b_2'$ can be defined as the configurations obtained from $(i,j)\in\hat S$ by detaching a particle from one corner towards north/south and east/west, respectively. Let $\hat{A} = \cup_{(i,j) \in \hat{S}} \hat{A}(i,j)$.

Given a starting state $(i,j) \in \hat{S}$, we define the following transition probabilities to the next robust configuration:
\begin{align}
    \label{trans_probs_b1} \hat{P}((i',j')|(i,j), b_1) &= \begin{cases}
            2/7, &\text{if } (i',j') = (i,j), \\
            5/7, &\text{if } (i',j')=(i,j+1), \\
            0, &\text{otherwise,}
        \end{cases} \quad \text{for }
        \begin{aligned}
            &i = 3, \ldots, L,\\
            &j = 2 \ldots L-3,
        \end{aligned}\\
    \label{trans_probs_b2} \hat{P}((i',j')|(i,j), b_2) &= \begin{cases}
            2/7, &\text{if } (i',j') = (i,j), \\
            5/7, &\text{if } (i',j')=(i+1,j), \\
            0, &\text{otherwise,}
    \end{cases} \quad \text{for }
    \begin{aligned}
        &i = 2, \ldots, L-3, \\
        &j = 3, \ldots L, 
    \end{aligned}\\
     \label{trans_probs_b1_prime1} \hat{P}((i',j')|(i,j), b_1') &= 
     \begin{cases}
            1/2, &\text{if } (i',j') = (i,j), \\
            1/2, &\text{if } (i',j')=(i,j+1),\\
            0, &\text{otherwise,}
    \end{cases} 
    \quad \text{for }
    \begin{aligned}
        &i = 2, \ldots L-2, \\
        &j = i \ldots L-3, 
    \end{aligned}\\
    \label{trans_probs_b1_prime2} \hat{P}((i',j')|(i,j), b_1') &= 
     \begin{cases}
            1/2, &\text{if } (i',j') = (i,j), \\
            1/3, &\text{if } (i',j')=(i,j+1), \\
            1/6, &\text{if } (i',j')=(i-1,j+1),\\
            0, &\text{otherwise,}
    \end{cases} \quad \text{for } 
    \begin{aligned}
        &i = 2 \ldots L-2, \\
        &j = 2, \ldots, i-1, 
    \end{aligned}\\
    \label{trans_probs_b2_prime1} \hat{P}((i',j')|(i,j), b_2') &= 
     \begin{cases}
            1/2, &\text{if } (i',j') = (i,j), \\
            1/2, &\text{if } (i',j')=(i+1,j), \\
            0, &\text{otherwise,}
    \end{cases} \quad \text{for }
     \begin{aligned}
         &i = j, \ldots, L-3, \\
         &j = 2, \ldots, L-2,
    \end{aligned}\\
    \label{trans_probs_b2_prime2} \hat{P}((i',j')|(i,j), b_2') &= 
     \begin{cases}
            1/2, &\text{if } (i',j') = (i,j), \\
            1/3, &\text{if } (i',j')=(i+1,j), \\
            1/6, &\text{if } (i',j')=(i+1,j-1),\\
            0, &\text{otherwise,}
    \end{cases} \quad \text{for }
    \begin{aligned}
        &i = 2, \ldots, j-1, \\
        &j = 2, \ldots, L-2, \\
    \end{aligned}\\
    \label{trans_probs_b1_Lmin2} \hat{P}((i',j')|(i,L-2), b_1) &= 
    \begin{cases}
        1/7, &\text{if } (i',j') = (i,L-2), \\
        6/7, &\text{if } (i',j') = (i,L), \\
        0, &\text{otherwise,}
    \end{cases} \quad i = 3, \ldots, L, \\
    \label{trans_probs_b2_Lmin2} \hat{P}((i',j')|(L-2, j), b_2) &= 
    \begin{cases}
        1/7, &\text{if } (i',j') = (L-2,j), \\
        6/7, &\text{if } (i',j') = (L,j), \\
        0, &\text{otherwise,}
    \end{cases} \quad j = 3, \ldots, L, \\ 
    \label{trans_probs_b1_prime_Lmin2} \hat{P}((i',j')|(i,L-2), b_1') &= 
    \begin{cases}
        1/3, &\text{if } (i',j') = (i,L-2), \\
        2/3, &\text{if } (i',j') = (i,L), \\
        0, &\text{otherwise,}
    \end{cases} \quad i = 2, \ldots, L-2, \\
    \label{trans_probs_b2_prime_Lmin2} \hat{P}((i',j')|(L-2, j), b_2') &= 
    \begin{cases}
        1/3, &\text{if } (i',j') = (L-2,j), \\
        2/3, &\text{if } (i',j') = (L,j), \\
        0, &\text{otherwise,}
    \end{cases} \quad j = 2, \ldots, L-2.
\end{align}

    \begin{figure}
\centering
\includegraphics[width=0.3\linewidth]{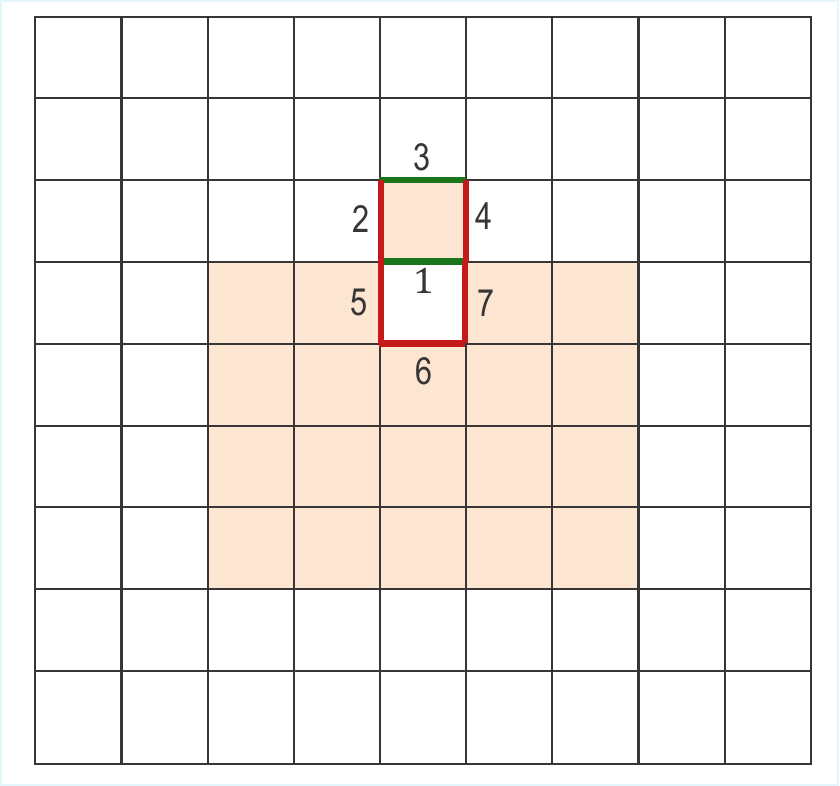}
\hspace{2mm}
 \includegraphics[width=0.3\linewidth]{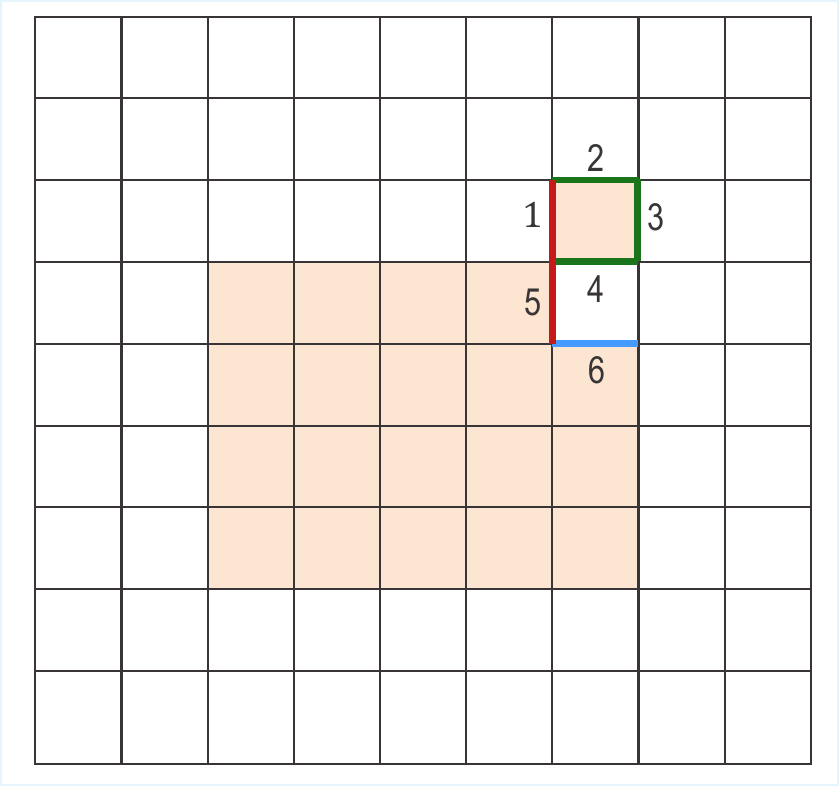}
   \\ \vspace{2mm}
   \includegraphics[width=0.3\linewidth]{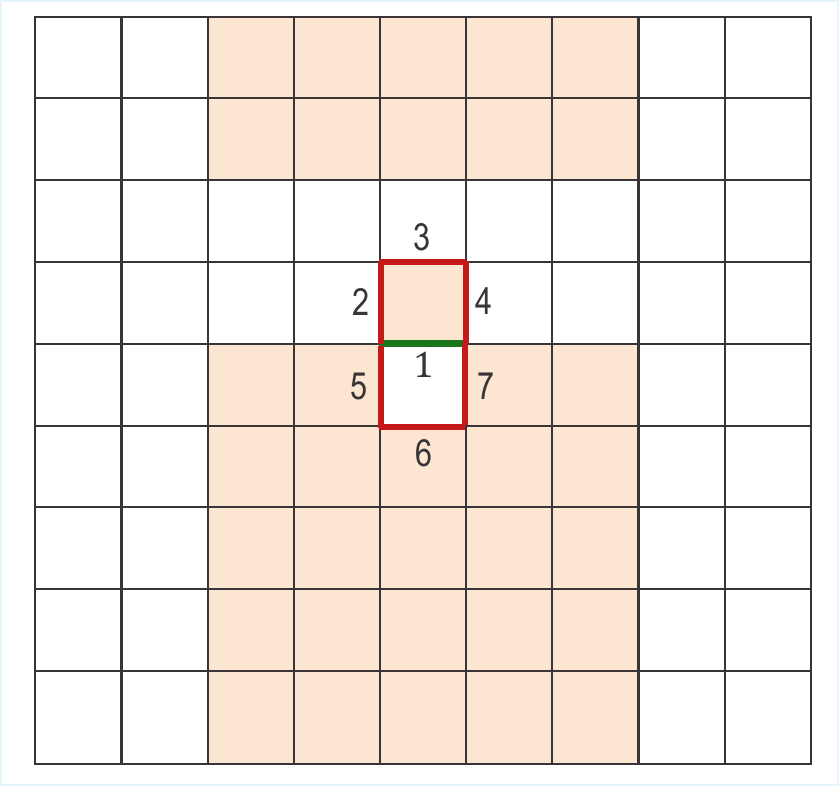}
\hspace{2mm}
    \includegraphics[width=0.3\linewidth]{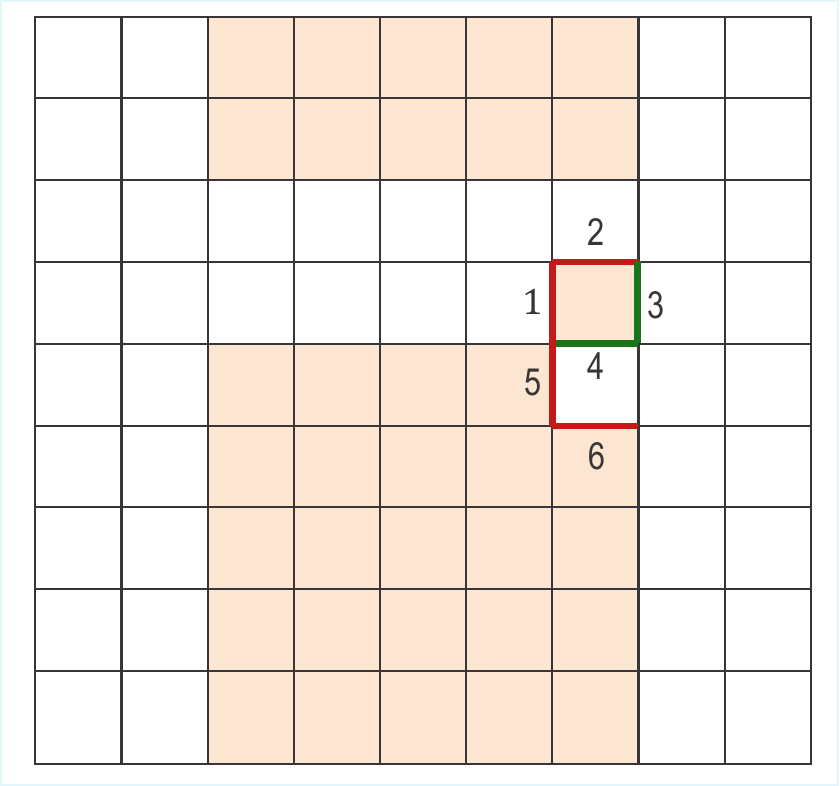}
\caption
    {\small Post-decision configurations after taking actions $b_1$ and $b_1'$ from a state $(i,j)$, $j < L-2$ (first row) and from state $(i, L-2)$ (second row).}
   \label{fig_trans_probs}
\end{figure}

In what follows, we justify expressions \eqref{trans_probs_b1}, \eqref{trans_probs_b1_prime1}, \eqref{trans_probs_b1_prime2}, \eqref{trans_probs_b1_Lmin2} and \eqref{trans_probs_b1_prime_Lmin2} corresponding to actions $b_1$ and $b_1'$. Analogous reasoning directly yields expressions \eqref{trans_probs_b2},  \eqref{trans_probs_b2_prime1}, \eqref{trans_probs_b2_prime2}, \eqref{trans_probs_b2_Lmin2}, \eqref{trans_probs_b2_prime_Lmin2}. 

   We first consider expression (\ref{trans_probs_b1}). After taking action $b_1$, there are seven susceptible bonds, as shown in Figure \ref{fig_trans_probs} (upper left). In the figure, the susceptible bonds are highlighted and labeled from 1 to 7. These are the only bonds whose selection leads to an interchange of particles at zero cost. Note that each of these bonds has an equal probability of being the first selected by the Kawasaki dynamics within this set. Recall that, after the first interchange of particles, any particles that become detached from the cluster are removed. Also, the decision maker can insert new particles only if the attachment of these particles leads to an energy decrease of $2U$. Hence, if one of the bonds labeled by 1 or 3 (colored green) in Figure \ref{fig_trans_probs} (upper left) is selected first, the system returns to state $(i,j)$. If instead one of the bonds labeled by 2, 4, 5, 6, or 7 (colored red) is selected first, new particles can be inserted and attached to the cluster until a rectangle of size $i \times (j+1)$ is formed. From this argument, we obtain expression (\ref{trans_probs_b1}).  

   Now consider expressions (\ref{trans_probs_b1_prime1}) and (\ref{trans_probs_b1_prime2}). Taking action $b_1'$ results in six susceptible bonds, as illustrated in Figure \ref{fig_trans_probs} (upper right). If one of the bonds labeled 2, 3 or 4 (colored green) is selected first, the system returns to the original rectangle. If instead one of the bonds 1 or 5 (colored red) is selected, new particles attach to the cluster to form a rectangle of size $i \times (j+1)$. If bond 6 is selected first, two possible evolutions may occur. Either new particles are inserted immediately, yielding a rectangle of size $i \times (j+1)$, or the vertical bar containing the moved particle first slides around the corner before particle insertion. This sliding mechanism is allowed only if $j < i$. Accordingly, if $j \geq i$, particles are inserted immediately, resulting in a rectangle of size $i \times (j+1)$. If $j < i$, the vertical bar first slides around the corner, after which particles are inserted, yielding a rectangle of size $(i-1) \times (j+1)$.  

   Finally, a similar argument leads to \eqref{trans_probs_b1_Lmin2} and \eqref{trans_probs_b1_prime_Lmin2}, see Figure \ref{fig_trans_probs} (lower row). The reward functions $\hat{r}_1: \hat{S} \times \hat{A} \rightarrow \mathbb{R}$ and $\hat{r}_2: \hat{S} \times \hat{A} \rightarrow \mathbb{R}$ are defined analogously to expressions (\ref{eq:rew1}) and (\ref{eq:rew2}), i.e.,
   \begin{align}
    \hat{r}_1((i,j), a) &= \begin{cases}
        1, &\text{if } (i,j) = (L, L), \\
        0, &\text{otherwise,} \label{eq:rew1_aux}
    \end{cases} \quad (i,j) \in \hat{S}, \quad a \in \hat{A},\\
    \hat{r}_2((i,j), a) &= -[H(\eta^{a'}) - H(\eta)], \quad (i,j) \in \hat{S}, \quad a \in \hat{A},  \label{eq:rew2_aux}
\end{align}
where $\eta$ is any robust configuration corresponding to state $(i,j)$ and $a'$ is any bond corresponding to $a$ in this configuration. Note that the energy difference defining $\hat{r}_2$ is independent of the specific choices of $\eta$ and $a'$ within the equivalence classes induced by $(i,j)$ and $a$, and hence $\hat{r}_2$ well-defined. 

\subsection{Main results}
    In this section, we state the main results of the paper. 

    \begin{thm}\label{thm1}
     Given the reward function $\hat{r}_1: \hat{S} \times \hat{A} \rightarrow \mathbb{R}$, as defined in expression \eqref{eq:rew1_aux}, a stationary and deterministic policy $\pi^* = (d^*)^{\infty}$ is optimal in the auxiliary MDP if and only if $d^*(i,j) \in A^*(i,j)$ for all $(i,j) \in \hat{S}$, where, letting $P(\hat{A})$ denote the power set of the action space $\hat{A}$, the function $A^*: \hat{S} \rightarrow P(\hat{A})$ is defined as
    \[
        A^*(i,j) = 
        \begin{cases}
            \{b_1\}, &\text{if } i = L, \quad j = 2, \ldots, L-2, \\
            \{b_2\}, &\text{if } i = 2, \ldots, L-2, \quad j = L, \\
             \{b_1, b_2\}, &\text{if } i,j = 2, \ldots, L-2.
        \end{cases}
    \]
\end{thm}

\begin{thm}\label{thm2}
 Given the reward function $\hat{r}_2: \hat{S} \times \hat{A} \rightarrow \mathbb{R}$, as defined in expression \eqref{eq:rew2_aux}, a stationary and deterministic policy $\pi^* = (d^*)^{\infty}$ is optimal in the auxiliary MDP if and only if $d^*(i,j) \in A^*(i,j)$ for all $(i, j) \in \hat{S}$, where, letting $P(\hat{A})$ denote the power set of the action space $\hat{A}$, the function $A^*: \hat{S} \rightarrow P(\hat{A})$ is defined as
\begin{equation*}
    A^*(i,j) = \begin{cases}
        \{b_1\}, &\text{if } i = L, \quad j = 2, \ldots, L-2, \\
        \{b_2\}, &\text{if } i = 2, \ldots, L-2, \quad j = L, \\
        \{b_1'\}, &\text{if } i = L-2, \quad j = 2, \ldots, L-2, \\
        \{b_2'\}, &\text{if } i = 2, \ldots, L-2,  \quad j = L-2, \\
        \{b_1', b_2'\}, &\text{otherwise.}
    \end{cases}
\end{equation*}
\end{thm}

It is apparent from Theorems \ref{thm1} and \ref{thm2} that the two reward functions induce optimal strategies with fundamentally different structural properties. The function $\hat{r}_1$, which rewards only reaching the full-box, leads to a strategy that prescribes interchanging particles along bonds that are not located at the corners of the rectangle. By contrast, the reward function $r_2$, which is based on the energy cost of particle interchanges, yields an optimal strategy that prescribes particle interchanges along corner bonds. 
This distinction is consistent with the underlying transition probabilities and energy costs. After selecting a non-corner bond, the probability of gaining an additional slice of particles is higher than after selecting a corner bond, which justifies the strategy induced by $\hat{r}_1$. Conversely, interchanging particles along a corner bond incurs a lower energy cost than interchanging along a non-corner bond, a fact that is reflected in the optimal strategy induced by $\hat{r}_2$.

\section{Proof of the main results}
\label{sec:proof}

\subsection{Proof of Theorem \ref{thm1}}

    Let $\pi = (d)^{\infty}$, where $d(i,j) \in A^*(i,j)$. It follows that $v^{\pi}_{\lambda}:\hat{S} \times \hat{A} \rightarrow \mathbb{R}$ satisfies the following recursive expressions:
    \begin{align}
        \label{pi1_rec1} &v^{\pi}_{\lambda}(L, L) = \dfrac{1}{1-\lambda}, \\
        \label{pi1_rec2} &v^{\pi}_{\lambda}(L, L-2) = \dfrac{6\lambda}{7-\lambda}v^{\pi}(L, L), \\
        \label{pi1_rec3} &v^{\pi}_{\lambda}(L-2, L-2) = \dfrac{6\lambda}{7-\lambda}v^{\pi}(L-2, L), \\
        \label{pi1_rec4} &v^{\pi}_{\lambda}(i, j) = \dfrac{5\lambda}{7-2\lambda}v^{\pi}(i, j+1), \quad i = 3, \ldots, L, \quad j = 3, \ldots, L-3, \quad i \geq j.
    \end{align}

    We show that for any $s = (i,j) \in \hat{S}$, we have
    \begin{equation}\label{Bellman_equations}
       \hat{r}_1(s, a) + \lambda \sum\limits_{s' \in \hat{S}} \hat{P}(s'|s, a)v^{\pi}_{\lambda}(s')< \hat{r}_1(s, a') + \lambda \sum\limits_{s' \in \hat{S}} \hat{P}(s'|s, a')v^{\pi}_{\lambda}(s'),
    \end{equation} 
    for all $a \in A^*(i,j)$, $a' \notin A^*(i,j)$ and 
    \begin{equation}\label{Bellman_equations_2}
        \hat{r}_1(s, a) + \lambda \sum\limits_{s' \in \hat{S}} \hat{P}(s'|s, a)v^{\pi}_{\lambda}(s')= \hat{r}_1(s, a') + \lambda \sum\limits_{s' \in \hat{S}} \hat{P}(s'|s, a')v^{\pi}_{\lambda}(s'),
    \end{equation}
    for all $a, a' \in A^*(i,j)$. 
    For states of the form $(L,j)$, $j = 2, \ldots, L-2$, the statement is obvious, since $b_1$ is the only available action. Equation \eqref{Bellman_equations} for the state $(L-2,L-2)$ becomes
    \[
    \dfrac{\lambda}{7}v^{\pi}_{\lambda}(L-2, L-2) + \dfrac{6\lambda}{7}v^{\pi}_{\lambda}(L, L-2) > \dfrac{\lambda}{3}v^{\pi}_{\lambda}(L-2, L-2) + \dfrac{2\lambda}{3}v^{\pi}_{\lambda}(L, L-2),
    \]
    which reduces to 
    \begin{equation}\label{ineqq1}
    v^{\pi}_{\lambda}(L, L-2) - v^{\pi}_{\lambda}(L-2, L-2) > 0.
    \end{equation}
For $(L-2, j)$, $j = 2, \ldots, L-3$, expressions (\ref{Bellman_equations}) and (\ref{Bellman_equations_2}) reduce to
\begin{align}
    \label{ineqq2}&v^{\pi}_{\lambda}(L-2, j+1) - v^{\pi}_{\lambda}(L-2,j) > 0, \\
    \label{ineqq3}&v^{\pi}_{\lambda}(L, j) - v^{\pi}_{\lambda}(L-2, j) > 0, \\
    \label{ineqq4}&v^{\pi}_{\lambda}(L-2,j) +5 v^{\pi}_{\lambda}(L-2, j+1) - 6v^{\pi}_{\lambda}(L, j) = 0.
\end{align}
For states $(i,i)$, $i = 2, \ldots, L-3$, we obtain
\begin{align}
    \label{ineqq5}&v^{\pi}_{\lambda}(i,i+1) - v^{\pi}_{\lambda}(i,i) > 0.
\end{align}
Finally, for states $(i,j)$, $i,j  = 2, \ldots, L-3$, $i > j$, expressions (\ref{Bellman_equations}) and (\ref{Bellman_equations_2}) can be written as:
\begin{align}
    \label{ineqq6}& -9v^{\pi}_{\lambda}(i,j) + 16v^{\pi}_{\lambda}(i, j+1) -7v^{\pi}_{\lambda}(i-1, j+1) > 0, \\
    \label{ineqq7}& v^{\pi}_{\lambda}(i+1,j) - v^{\pi}_{\lambda}(i,j) > 0, \\
    \label{ineqq8}& v^{\pi}_{\lambda}(i+1,j) = v^{\pi}_{\lambda}(i,j+1).
\end{align}

\paragraph{\textbf{Proof of expression (\ref{ineqq1})}:}
    Using expressions \eqref{pi1_rec1}--\eqref{pi1_rec3}, we obtain
    \begin{align}
        \label{pi1_LLmin2}v^{\pi}_{\lambda}(L, L-2) &= \dfrac{6\lambda}{(1-\lambda)(7-\lambda)},\\
        \label{pi1_Lmin2Lmin2}v^{\pi}_{\lambda}(L-2, L-2) &= \dfrac{36\lambda^2}{(1-\lambda)(7-\lambda)^2}.
    \end{align}   
The correctness of inequality (\ref{ineqq1}) follows directly from inserting these explicit expressions.
\vspace{5mm}
\paragraph{\textbf{Proof of expression (\ref{ineqq2})}:}
We prove the validity of expression (\ref{ineqq2}) by induction over $j$. Consider $j = L-3$. Using expressions (\ref{pi1_Lmin2Lmin2}) and (\ref{pi1_rec4}), we obtain
\begin{equation}\label{pi1_Lmin2Lmin3}
    v^{\pi}_{\lambda}(L-2, L-3) = \dfrac{180 \lambda^3}{(7-2\lambda)(1-\lambda)(7-\lambda)^2}.
\end{equation}
Inserting expressions (\ref{pi1_Lmin2Lmin2}) and (\ref{pi1_Lmin2Lmin3}) into inequality (\ref{ineqq2}) yields its correctness for $j = L-3$. Now, suppose that it holds for $j = k+1$, for some $k = 2, \ldots, L-4$. For $j = k$, we then obtain, using expression (\ref{rec_4})
\begin{equation*}
    v^{\pi}_{\lambda}(L-2, k+1) - v^{\pi}_{\lambda}(L-2,k) = \dfrac{5\lambda}{7-2\lambda}(v^{\pi}_{\lambda}(L-2, k+2) - v^{\pi}_{\lambda}(L-2, k+1)) > 0,
\end{equation*}
by the induction hypothesis. Thus, expression (\ref{ineqq2}) holds for all $j = 2, \ldots, L-3$.
\vspace{5mm}
\paragraph{\textbf{Proof of expression (\ref{ineqq3})}:}
We prove inequality (\ref{ineqq3}) again by means of induction over $j$. Consider $j = L-3$. Using expression (\ref{rec_4}), we obtain
\begin{equation}\label{pi1_LLmin3}
    v^{\pi}_{\lambda}(L, L-3) = \dfrac{30\lambda^2}{(7-2\lambda)(1-\lambda)(7-\lambda)}.
\end{equation}
The validity of expression (\ref{ineqq3}) for $j = L-3$ follows directly from inserting expressions (\ref{pi1_LLmin3}) and (\ref{pi1_Lmin2Lmin3}). We proceed to assume that it is correct for $j = k+1$, for some $k = 2, \ldots, L-4$. This yields for $j = k$, using expression (\ref{pi1_rec4}),
\begin{equation}
    v^{\pi}_{\lambda}(L,k) - v^{\pi}_{\lambda}(L-2,k) = \dfrac{5\lambda}{7-2\lambda}(v^{\pi}_{\lambda}(L, k+1) - v^{\pi}_{\lambda}(L-2, k+1)) > 0.
\end{equation}
This establishes the correctness of expression (\ref{ineqq3}) for all $j = 3, \ldots, L-3$.
\vspace{5mm}
\paragraph{\textbf{Proof of expression (\ref{ineqq4})}:}
First, we verify the validity of inequality (\ref{ineqq4}) for $j = L-3$ using expressions (\ref{pi1_Lmin2Lmin3}), (\ref{pi1_Lmin2Lmin2}) and (\ref{pi1_LLmin3}). Now, we assume that the inequality holds for $j = k+1$, for some $k = 2, \ldots, L-4$. This induction hypothesis and expression (\ref{pi1_rec4}) yield
\begin{equation*}
\begin{array}{ll}
    v^{\pi}_{\lambda}(L-2, k) + 5v^{\pi}_{\lambda}(L-2, k+1) - 6 v^{\pi}_{\lambda}(L, k) \\
    \qquad \qquad \qquad = \dfrac{5\lambda}{7-2\lambda}(v^{\pi}_{\lambda}(L-2, k+1) + 5v^{\pi}_{\lambda}(L-2, k+2) - 6 v^{\pi}_{\lambda}(L, k+1)) > 0.
    \end{array}
\end{equation*}
It follows that inequality (\ref{ineqq4}) holds for all $j = 3, \ldots, L-3$.
\vspace{5mm}
\paragraph{\textbf{Proof of expression (\ref{ineqq5})}:}
Using expressions (\ref{pi1_rec4}) and (\ref{pi1_Lmin3Lmin3}), we obtain
\begin{equation}\label{pi1_Lmin3Lmin3}
    v^{\pi}_{\lambda}(L-3, L-3) = \dfrac{900 \lambda^4}{(7-2\lambda)^2(1-\lambda)(7-\lambda)^2}.
\end{equation}
Inequality (\ref{ineqq5}) is easily verified for $i = L-3$ using expressions (\ref{pi1_Lmin3Lmin3}) and (\ref{pi1_Lmin2Lmin3}). We now assume that it holds for $i = k+1$, for some $k = 2, \ldots, L-4$. This implies for $i = k$, invoking expression (\ref{pi1_rec4}) twice and using symmetry,
\begin{align*}
    v^{\pi}_{\lambda}(k, k+1) - v^{\pi}_{\lambda}(k, k) &= v^{\pi}_{\lambda}(k+1, k) - v^{\pi}_{\lambda}(k, k) \\
    &= \dfrac{5\lambda}{7-2\lambda} [v^{\pi}_{\lambda}(k+1, k+1) - v^{\pi}_{\lambda}(k, k+1)] \\
    &= \dfrac{5\lambda}{7-2\lambda} [v^{\pi}_{\lambda}(k+1, k+1) - v^{\pi}_{\lambda}(k+1, k)] \\
    &= \left(\dfrac{5\lambda}{7-2\lambda}\right)^2[v^{\pi}_{\lambda}(k+1, k+2) - v^{\pi}_{\lambda}(k+1, k+1)].
\end{align*}
This establishes the validity of inequality (\ref{ineqq5}) for all $i = 3, \ldots, L-3$. 
\vspace{5mm}
\paragraph{\textbf{Proof of expression (\ref{ineqq6})}:}
We prove the correctness of inequality (\ref{ineqq6}) by means of an induction argument with the following structure: first, we show that the inequality holds for states of the form $(i, i-1)$, for $i = 3, \ldots, L-3$, by means of induction over $i$. Then, for each $i = 3, \ldots, L-3$, we show that the inequality is true for states $(i, k+1)$, $k = 2, \ldots, i-2$. 

Consider states of the form $(i, i-1)$, where $i = 3, \ldots, L-3$. To prove that inequality (\ref{ineqq6}) holds for these states, we first verify its correctness for state $(L-3, L-4)$. Using expressions (\ref{pi1_Lmin3Lmin3}) and (\ref{rec_4}), we obtain
\begin{equation}\label{pi1_Lmin3Lmin4}
    v^{\pi}_{\lambda}(L-3, L-4) = \dfrac{4500\lambda^5}{(7-2\lambda)^3(1-\lambda)(7-\lambda)^2}.
\end{equation}
Inserting the explicit expressions (\ref{pi1_Lmin3Lmin3}) and (\ref{pi1_Lmin3Lmin4}) into inequality (\ref{ineqq6}) yields its validity for state $(L-3, L-4)$. Now, we assume that the inequality holds for state $(k+1, k)$, for some $k = 2, \ldots, L-4$. Invoking expression (\ref{rec_4}) twice and using symmetry, we obtain
\begin{align*}
    &-9v^{\pi}_{\lambda}(k, k-1) + 16v^{\pi}_{\lambda}(k, k) -7v^{\pi}_{\lambda}(k-1, k) \\
    &\qquad\qquad\qquad= \dfrac{5\lambda}{7-2\lambda}(-9v^{\pi}_{\lambda}(k, k) + 16v^{\pi}_{\lambda}(k, k+1) -7v^{\pi}_{\lambda}(k-1, k+1)) \\
    &\qquad\qquad\qquad= \dfrac{5\lambda}{7-2\lambda}(-9v^{\pi}_{\lambda}(k, k) + 16v^{\pi}_{\lambda}(k+1, k) - 7 v^{\pi}_{\lambda}(k+1, k-1)) \\
    &\qquad\qquad\qquad= \left(\dfrac{5\lambda}{7-2\lambda}\right)^2(-9v^{\pi}_{\lambda}(k, k+1) + 16v^{\pi}_{\lambda}(k+1, k+1) - 7v^{\pi}_{\lambda}(k+1, k) \\
    &\qquad\qquad\qquad= \left(\dfrac{5\lambda}{7-2\lambda}\right)^2(-9v^{\pi}_{\lambda}(k+1, k) + 16v^{\pi}_{\lambda}(k+1, k+1) -7v^{\pi}_{\lambda}(k, k+1)) > 0,
\end{align*}
by the induction hypothesis. Thus, inequality (\ref{ineqq6}) holds for all states $(i, i-1)$, $i = 3, \ldots, L-3$. 

Given $i = 3, \ldots, L-3$, we now show that inequality (\ref{ineqq6}) is true for states $(i, j)$, $j = 3, \ldots, i-1$. Note that the case $j = i-1$ has been established above. We now assume that the inequality is correct for state $(i, k+1)$, for some $k = 2, \ldots, i-2$. This implies for state $(i,k)$, using expression (\ref{rec_4})
\begin{align*}
    &-9v^{\pi}_{\lambda}(i,k) + 16v^{\pi}_{\lambda}(i, k+1) -7v^{\pi}_{\lambda}(i-1, k+1) \\
    &\qquad\qquad=\dfrac{5\lambda}{7-2\lambda}(-9v^{\pi}_{\lambda}(i, k+1) + 16v^{\pi}_{\lambda}(i, k+2) - 7v^{\pi}_{\lambda}(i-1, k+2)) > 0.
\end{align*}
Hence, we established the validity of inequality (\ref{ineqq6}) for all $(i,j)$, $i,j = 3, \ldots, L-3$, $i > j$. 
\vspace{5mm}
\paragraph{\textbf{Proof of expression (\ref{ineqq7})}:}
For all $(i,j)$, $i,j = 3, \ldots, L-3$, $i > j$, we have, using symmetry and expression (\ref{rec_4}),
\begin{align*}
    v^{\pi}_{\lambda}(i + 1, j) - v^{\pi}_{\lambda}(i,j) &= v^{\pi}_{\lambda}(i+1,j) - v^{\pi}_{\lambda}(j, i) \\
    &= v^{\pi}_{\lambda}(i+1, j) - \dfrac{5\lambda}{7-2\lambda}v^{\pi}_{\lambda}(j, i+1) \\
    &= \dfrac{7(1-\lambda)}{7-2\lambda}v^{\pi}_{\lambda}(i+1,j) > 0.
\end{align*}
\vspace{5mm}
\paragraph{\textbf{Proof of expression (\ref{ineqq8})}:}
Using symmetry and expression (\ref{rec_4}), we obtain
\begin{equation*}
    v^{\pi}_{\lambda}(i+1, j) = \dfrac{5\lambda}{7-2\lambda} v^{\pi}_{\lambda}(i+1, j+1) = \dfrac{5\lambda}{7-2\lambda}v^{\pi}_{\lambda}(j+1, i+1) = v^{\pi}_{\lambda}(j+1, i) = v^{\pi}_{\lambda}(i, j+1)
\end{equation*}
for all $i, j = 3, \ldots, L-3$, $i > j$, thereby establishing expression (\ref{ineqq8}).

\subsection{Proof of Theorem \ref{thm2}}

    By Theorem \ref{optimality_thm}, it suffices to show that the value function $v^{\pi}_{\lambda}: \hat{S} \rightarrow \mathbb{R}$ of a stationary, deterministic policy $\pi$ satisfies the Bellman equations if and only if it has the specified form. We will prove this statement for states of the form $(i,j) \in \hat{S}$, $i \geq j$. The analogous result for states of the form $(i,j) \in \hat{S}$, $i < j$, will follow from symmetry.

    Consider a policy $\pi = (d)^{\infty}$ with $d(i,j) \in A^*(i,j)$ for all $(i,j) \in \hat{S}$. Using the transition probabilities \eqref{trans_probs_b1}--\eqref{trans_probs_b2_prime_Lmin2}, we derive the following expressions for the value function $v^{\pi}_{\lambda}: \hat{S} \times \hat{A} \rightarrow \mathbb{R}$. 
        \begin{align}
      \label{rec_1}&v_{\lambda}^{\pi}(L, L)= 0, \\
    \label{rec_2}&v^{\pi}_{\lambda}(L, L-2) = -\dfrac{21}{7-\lambda}U + \dfrac{6\lambda}{7-\lambda} v^{\pi}(L, L), \\ 
    \label{rec_3}&v^{\pi}_{\lambda}(L,j) = -\dfrac{21}{7-2\lambda}U+\dfrac{5\lambda}{7-2\lambda}v^{\pi}(L, j+1), \quad 3 \leq j \leq L-3,\\
    \label{rec_4}&v^{\pi}_{\lambda}(L-2, L-2) = -\dfrac{6}{3-\lambda}U+\dfrac{2\lambda}{3-\lambda}v^{\pi}(L, L-2), \\
    \label{rec_5}&v^{\pi}_{\lambda}(L-2,j) = -\dfrac{4}{2-\lambda}U+\dfrac{\lambda}{2-\lambda}v^{\pi}_{\lambda}(L-2, j+1), \quad 3 \leq j \leq L-3, \\
    \label{rec_6}&v^{\pi}_{\lambda}(i,j) = -\dfrac{4}{2-\lambda}U+\dfrac{\lambda}{2-\lambda}v^{\pi}_{\lambda}(i,j+1), \quad 3 \leq i \leq L-2, \quad 3 \leq j \leq i.
\end{align}
We proceed to show that 
\begin{equation}\label{inequalities}
    \hat{r}_2(s, a) + \lambda \sum\limits_{s' \in \hat{S}}\hat{P}(s'|s, a)v^{\pi}_{\lambda}(s') >  \hat{r}_2(s, a') + \lambda \sum\limits_{s' \in \hat{S}}\hat{P}(s'|s, a')v^{\pi}_{\lambda}(s'),
\end{equation}
for all $a \in A^*(i,j)$, $a' \notin A^*(i,j)$, $s = (i,j) \in \hat{S}$ and
\begin{equation}\label{equalities}
     \hat{r}_2(s, a) + \lambda \sum\limits_{s' \in \hat{S}}\hat{P}(s'|s, a)v^{\pi}_{\lambda}(s') =  \hat{r}_2(s, a') + \lambda \sum\limits_{s' \in \hat{S}}\hat{P}(s'|s, a')v^{\pi}_{\lambda}(s'),
\end{equation}
for all $a, a' \in A^*(i,j)$. Observe that the statement for states $(i,j) \in \hat{S}$ with $i = L$ or $j = L$ is obvious, since $\hat{A}(i,j)$ contains only one action in this case. 

For state $(L-2, L-2)$, inequality (\ref{inequalities}) becomes
\begin{equation*}
    -2U + \dfrac{\lambda}{3}v^{\pi}_{\lambda}(L-2, L-2) + \dfrac{2\lambda}{3}v^{\pi}_{\lambda}(L, L-2) > -3U + \dfrac{\lambda}{7}v^{\pi}_{\lambda}(L-2,L-2) + \dfrac{6\lambda}{7}v^{\pi}_{\lambda}(L, L-2),
\end{equation*}
which can be written as
\begin{equation}
    \label{ineq_1} 21U + 4\lambda v^{\pi}_{\lambda}(L-2, L-2) - 4\lambda v^{\pi}_{\lambda}(L, L-2) > 0.
\end{equation}
For the remaining states $(i,j) \in \hat{S}$, $i \geq j$, expressions (\ref{inequalities}) and (\ref{equalities}) reduce to
\begin{align}
     \label{ineq_2}&14U+3\lambda v^{\pi}_{\lambda}(L-2, j) -3\lambda v^{\pi}_{\lambda}(L-2, j+1) > 0, \quad 3 \leq j \leq L-3, \\
     \label{ineq_3}&14U + 5\lambda v^{\pi}_{\lambda}(L-2, j) + 7\lambda v^{\pi}_{\lambda}(L-2, j+1) - 12\lambda v^{\pi}_{\lambda}(L,j) > 0, \quad 3 \leq j \leq L-3, \\
     \label{ineq_4}&v^{\pi}_{\lambda}(L-2, j) + 3v^{\pi}_{\lambda}(L-2, j+1) - 4 v^{\pi}_{\lambda}(L,j) > 0, \quad 3 \leq j \leq L-3, \\
     \label{ineq_5}&v^{\pi}_{\lambda}(i, j+1) = v^{\pi}_{\lambda}(i+1, j), \quad 3 \leq j \leq i \leq L-3, \\
     \label{ineq_6}&14U + 3\lambda v^{\pi}_{\lambda}(i,j) - 3\lambda v^{\pi}_{\lambda}(i, j+1) > 0, \quad 3 \leq j \leq i \leq L-3.
\end{align}
We proceed to prove each of the expressions \eqref{ineq_1}--\eqref{ineq_6}.
\vspace{5mm}
\paragraph{\textbf{Proof of expression (\ref{ineq_1}):}}
Using expressions (\ref{rec_1}), (\ref{rec_2}) and (\ref{rec_4}), we obtain
\begin{equation}\label{LLmin2_expl}
    v^{\pi}_{\lambda}(L, L-2) = -\dfrac{21}{7-\lambda}U
\end{equation}
and
\begin{equation}\label{Lmin2Lmin2_expl}
    v^{\pi}_{\lambda}(L-2, L-2) = - \dfrac{6(7+6\lambda)}{(7-\lambda)(3-\lambda)}U.
\end{equation}
Inserting these explicit expressions in inequality (\ref{ineq_1}) establishes its validity.
\vspace{5mm}
\paragraph{\textbf{Proof of expression (\ref{ineq_2}):}}
First, we verify the validity of expression (\ref{ineq_2}) for $j = L-3$. Using expressions (\ref{rec_5}) and (\ref{Lmin2Lmin2_expl}), we obtain
\begin{equation}\label{Lmin2Lmin3_expl}
    v^{\pi}_{\lambda}(L-2, L-3) = -\dfrac{20\lambda^2 + \lambda + 42}{(7-\lambda)(3-\lambda)(2-\lambda)}.
\end{equation}
Inserting expressions (\ref{Lmin2Lmin3_expl}) and (\ref{Lmin2Lmin2_expl}) proves the validity of expression (\ref{ineq_2}) for $j = L-3$. 
Now, suppose that it holds for some $j = k+1$, $k = 2, \ldots, L-4$. Invoking expressions (\ref{rec_3}) and (\ref{rec_4}) and the induction hypothesis now yields
\begin{align*}
    &14U+3\lambda v^{\pi}_{\lambda}(L-2, k) -3\lambda v^{\pi}_{\lambda}(L-2, k+1) \\
    &\qquad\qquad= 14U + \dfrac{3\lambda^2}{2-\lambda}v^{\pi}_{\lambda}(L-2, k+1) - \dfrac{3\lambda^2}{2-\lambda}v^{\pi}_{\lambda}(L-2, k+2) \\
    &\qquad\qquad= \dfrac{\lambda}{2-\lambda}[14U + 3\lambda v^{\pi}_{\lambda}(L-2, k+1) - 3\lambda v^{\pi}_{\lambda}(L-2, k+2)] + \dfrac{28(1-\lambda)}{2-\lambda} U > 0.
\end{align*}
It follows that inequality (\ref{ineq_2}) holds for all $j = 2, \ldots, L-3$.
\vspace{5mm}
\paragraph{\textbf{Proof of expression (\ref{ineq_3}):}}
Again, we start by verifying the correctness or inequality (\ref{ineq_3}) for $j = L-3$. Using expression (\ref{rec_3}) and (\ref{LLmin2_expl}), we obtain
\begin{equation}\label{LLmin3_expl}
    v^{\pi}_{\lambda}(L, L-3) = -\dfrac{21(7+4\lambda)}{(7-\lambda)(7-2\lambda)}.
\end{equation}
Inserting expressions (\ref{Lmin2Lmin2_expl}), (\ref{Lmin2Lmin3_expl}) and (\ref{LLmin3_expl}) into expression (\ref{ineq_3}) yields its correctness or $j = L-3$.
We proceed to make the assumption that it holds for some $j = k+1$, $k = 2, \ldots, L-4$. Using expressions (\ref{rec_3}) and (\ref{rec_5}), we now obtain
\begin{align*}
    &14U + 5\lambda v^{\pi}_{\lambda}(L-2, k) + 7\lambda v^{\pi}_{\lambda}(L-2, k+1) - 12\lambda v^{\pi}_{\lambda}(L,k) \\
    &\qquad\qquad= \dfrac{\lambda}{2-\lambda}[14U+5\lambda v^{\pi}_{\lambda}(L-2, k+1) + 7\lambda v^{\pi}_{\lambda}(L-2, k+2) - 12\lambda v^{\pi}_{\lambda}(L, k+1)] \\
    &\qquad\qquad\quad- \dfrac{4(25\lambda^2+21\lambda-49)}{(2-\lambda)(7-2\lambda)}U - \dfrac{36\lambda^2(1-\lambda)}{(2-\lambda)(7-2\lambda)}v^{\pi}_{\lambda}(L, k+1).
\end{align*}
The induction hypothesis now implies that it suffices to prove that
\begin{align}\label{ineq_3_help1}
    \dfrac{36\lambda^2(1-\lambda)}{(2-\lambda)(7-2\lambda)}v^{\pi}_{\lambda}(L, \ell+1) < - \dfrac{4(25\lambda^2+21\lambda-49)}{(2-\lambda)(7-2\lambda)}U,
\end{align}
for all $\ell = 2, \ldots, L-4$. We prove this statement by means of an embedded induction argument over $\ell$. After checking the validity of expression (\ref{ineq_3_help1}) for $\ell = L-4$, we assume that it holds for $\ell = \ell' + 1$, for some $\ell' = 2, \ldots, L-5$. Using this induction hypothesis, we obtain
\begin{align}
    \nonumber \dfrac{36\lambda^2(1-\lambda)}{(2-\lambda)(7-2\lambda)}v^{\pi}_{\lambda}(L, \ell'+1) &=  \dfrac{36\lambda^2(1-\lambda)}{(2-\lambda)(7-2\lambda)} \left[-\dfrac{21}{7-2\lambda} U + \dfrac{5\lambda}{7-2\lambda}v^{\pi}_{\lambda}(L, \ell'+2)\right] \\
    &< - \dfrac{4(25\lambda^2+21\lambda-49)}{(2-\lambda)(7-2\lambda)}U.
\end{align}
This establishes expression (\ref{ineq_3_help1}) for all $\ell = 2, \ldots, L-4$, which yields the correctness of expression (\ref{ineq_3}) for all $j = 3, \ldots, L-3$. 
\vspace{5mm}
\paragraph{\textbf{Proof of expression (\ref{ineq_4}):}}
Again, we start by verifying the validity of expression (\ref{ineq_4}) for $j = L-3$. This is easily done by inserting expressions (\ref{Lmin2Lmin2_expl}), (\ref{Lmin2Lmin3_expl}) and (\ref{LLmin3_expl}). Suppose now that it holds for $j = k+1$, $k = 2, \ldots, L-4$. Invoking expressions (\ref{rec_3}) and (\ref{rec_5}) yields
\begin{align}
    \nonumber &v^{\pi}_{\lambda}(L-2, k) + 3v^{\pi}_{\lambda}(L-2, k+1) - 4 v^{\pi}_{\lambda}(L,k) \\
    \nonumber &\qquad\qquad= \dfrac{\lambda}{2-\lambda} \left[v^{\pi}_{\lambda}(L-2, k+1) + 3v^{\pi}_{\lambda}(L-2, k+2) - 4v^{\pi}_{\lambda}(L, k+1)\right] \\
    \nonumber &\qquad\qquad\quad+ \dfrac{56-52\lambda}{2\lambda^2-11\lambda + 14} U - \dfrac{12\lambda (1-\lambda)}{(2-\lambda)(7-2\lambda)}v^{\pi}_{\lambda}(L, k+1).
\end{align}
By the induction hypothesis, it now suffices to show that
\begin{align}\label{ineq_4_help1}
    \dfrac{12\lambda (1-\lambda)}{(2-\lambda)(7-2\lambda)}v^{\pi}_{\lambda}(L, \ell+1) < \dfrac{56-52\lambda}{2\lambda^2-11\lambda + 14} U,
\end{align}
for all $\ell = 2, \ldots, L-4$. Again, we use an embedded induction argument over $\ell$ to establish this statement. First, we verify the correctness of expression (\ref{ineq_4_help1}) for $\ell = L-4$, using expression (\ref{LLmin3_expl}). Now, suppose that it holds for $\ell = \ell'+1$, for some $\ell' = 2, \ldots, L-5$. Using this induction hypothesis, together with expression (\ref{rec_3}), we obtain
\begin{align}
    \nonumber \dfrac{12\lambda (1-\lambda)}{(2-\lambda)(7-2\lambda)}v^{\pi}_{\lambda}(L, \ell'+1) &= \dfrac{12\lambda (1-\lambda)}{(2-\lambda)(7-2\lambda)} \left[-\dfrac{21}{7-2\lambda} U + \dfrac{5\lambda}{7-2\lambda}v^{\pi}_{\lambda}(L, \ell'+2)\right] \\
    &< \dfrac{56-52\lambda}{2\lambda^2-11\lambda + 14} U. 
\end{align}
This shows the validity of expression (\ref{ineq_4_help1}) for $\ell = 2, \ldots, L-4$ and thus the correctness of expression (\ref{ineq_4}) for all $j = 3, \ldots, L-3$.
\vspace{5mm}
\paragraph{\textbf{Proof of expression (\ref{ineq_5}):}}
First, observe that the case $(i,i)$, $i = 3, \ldots, L-3$ follows immediately by symmetry. Now, given $i = 3, \ldots, L-3$, we assume that it holds for state $(i, k+1)$ for some $k = 2, \ldots, L-4$. Using this induction hypothesis and expression (\ref{rec_6}), we obtain
\begin{align*}
    v^{\pi}_{\lambda}(i, k+1) &= -\dfrac{4}{2-\lambda}U + \dfrac{\lambda}{2-\lambda}v^{\pi}_{\lambda}v^{\pi}_{\lambda}(i, k+2) \\
    &= -\dfrac{4}{2-\lambda}U + \dfrac{\lambda}{2-\lambda}v^{\pi}_{\lambda}(i+1, k+1) = v^{\pi}_{\lambda}(i+1, k).
\end{align*}
This concludes the proof of expression (\ref{ineq_5}).
\vspace{5mm}
\paragraph{\textbf{Proof of expression (\ref{ineq_6}):}}
We start by showing that the expression is correct for states of the form $(i,i)$, $i = 3, \ldots, L-3$. Using expression (\ref{rec_6}), we can write
\begin{align*}
    &14U + 3\lambda v^{\pi}_{\lambda}(i,i) - 3\lambda v^{\pi}_{\lambda}(i,i+1) \\
    &\qquad\qquad\qquad\qquad= 14U + 3\lambda \left[\dfrac{4}{2-\lambda}U + \dfrac{\lambda}{2-\lambda}v^{\pi}_{\lambda}(i,i+1)\right] - 3\lambda v^{\pi}_{\lambda}(i, i+1) \\
    &\qquad\qquad\qquad\qquad= \dfrac{2(14-\lambda)}{2-\lambda}U - \dfrac{6\lambda(1-\lambda)}{2-\lambda} v^{\pi}_{\lambda}(i, i+1).
\end{align*}
It follows that expression (\ref{ineq_6}) holds for $(i,i)$, $i = 3, \ldots, L-3$ if 
\begin{align}\label{ineq_6_help1}
    v^{\pi}_{\lambda}(i, i+1) < \dfrac{14-\lambda}{3\lambda(1-\lambda)}U.
\end{align}
We proceed to show that expression (\ref{ineq_6_help1}) is valid for $i = 3, \ldots, L-3$ by means of induction over $i$. First, we verify its correctness for $i = L-3$, using expression (\ref{Lmin2Lmin3_expl}). Now, suppose that it is true for $i = k+1$, $k = 2, \ldots, L-4$. We show that this implies its validity for $i = k$. By symmetry, expression (\ref{rec_6}) and the induction hypothesis, we obtain
\begin{align*}
    v^{\pi}_{\lambda}(k, k+1) &= \dfrac{4}{2-\lambda} U + \dfrac{\lambda}{2-\lambda} v^{\pi}_{\lambda}(k, k+2) \\
    &= \dfrac{4}{2-\lambda} U + \dfrac{\lambda}{2-\lambda} v^{\pi}_{\lambda}(k+2,k) \\
    &= \dfrac{4}{2-\lambda}U + \dfrac{\lambda}{2-\lambda} \left[\dfrac{4}{2-\lambda} U + \dfrac{\lambda}{2-\lambda} v^{\pi}_{\lambda}(k+2, k+1)\right] \\
    &= \dfrac{8}{(2-\lambda)^2}U + \dfrac{\lambda^2}{(2-\lambda)^2}v^{\pi}_{\lambda}(k+1, k+2) \\
    &< \dfrac{12+\lambda}{3(2-\lambda)(1-\lambda)} U < \dfrac{14-\lambda}{3\lambda(1-\lambda)} U.
\end{align*}
This establishes the correctness of expression (\ref{ineq_6}) for states of the form $(i,i)$, $i = 3, \ldots, L-3$. Now, given $i = 3, \ldots, L-3$, we assume that expression (\ref{ineq_6}) is valid for state $(i,k+1)$, for some $k = 2, \ldots, L-4$. For state $(i,k)$, this induction hypothesis implies, invoking expression (\ref{rec_6}),
\begin{align*}
    &14U + 3\lambda v^{\pi}_{\lambda}(i,k) - 3\lambda v^{\pi}_{\lambda}(i,k+1) \\
    &\qquad\qquad= 14U + 3\lambda \left[\dfrac{4}{2-\lambda} U + \dfrac{\lambda}{2-\lambda} v^{\pi}_{\lambda}(i,k+1)\right] - 3\lambda \left[\dfrac{4}{2-\lambda}U + \dfrac{\lambda}{2-\lambda} v^{\pi}_{\lambda}(i,k+2)\right] \\
    &\qquad\qquad= \dfrac{\lambda}{2-\lambda} \left[14U + 3\lambda v^{\pi}_{\lambda}(i,k+1) - 3\lambda v^{\pi}_{\lambda}(i, k+2) \right] + \dfrac{28(1-\lambda)}{2-\lambda}U > 0.
\end{align*}
Thus, it follows that expression (\ref{ineq_6}) holds for all $(i,j)$, $i,j = 3, \ldots, L-3$.  

\section{Conclusions and future work}
\label{sec:concl}
In this paper, we introduced an MDP perspective on the low--temperature Ising model evolving according to Kawasaki dynamics inside a finite box in the metastable regime. We formulated a controlled version of the dynamics, specified the associated state and action spaces, and considered two distinct reward functions. We then studied a reduced MDP restricted to configurations containing a single cluster of particles, allowing us to identify the exact form of the optimal policies for both reward functions. The comparison between the two settings illustrates how adding an energy--based cost alters the preferred growth mechanism.

There are several directions that we believe deserve further investigation. A natural step is to relax some of the simplifying assumptions of the reduced model, for instance by allowing multiple clusters. It would also be interesting to quantify how closely the optimal controlled dynamics mimic the behavior of the original Kawasaki dynamics at low temperature, and to understand which features of metastable transitions can be captured through suitable choices of reward functions.

Another possibility is to take advantage of the geometric structure that naturally arises in Kawasaki dynamics. Since particle exchanges occur locally, the behavior of the system is largely determined by the shape and local geometry of the particle cluster, namely, whether a boundary site is flat, protruding, or part of a corner, and how many occupied neighbours it has. This is precisely the type of information that drives energy changes in the original dynamics. In the simplified MDP we considered, we already made use of some of these geometric features; however, capturing the full range of possible boundary shapes would lead to an extremely rich, and analytically difficult, policy space. For this reason, one promising direction is to focus on policies that depend only on coarse geometric descriptions of the boundary, for instance the distribution of different boundary types, rather than on the full microscopic configuration. Such reduced descriptions retain the essential mechanisms that influence nucleation and growth, while avoiding the combinatorial explosion that comes with treating every possible configuration separately. Such geometric reductions open a natural path for future work: they would extend the ideas developed here while allowing a more detailed account of the local structures that promote or hinder growth, and of how different reward structures interact with these features. This direction would build on the present paper by connecting the MDP viewpoint closer to the full geometric complexity of Kawasaki dynamics.

On the computational side, one could explore approximate dynamic programming or reinforcement learning approaches to deal with the rapid growth of the state space. Even in the simplified setting considered here, the number of possible configurations increases quickly with the box size, and computing exact value functions or optimal policies becomes impractical. Techniques such as 
policy iteration with function approximation could help identify near--optimal strategies without having to enumerate the full state space. Such computational tools may offer a practical way to study controlled metastable behavior in settings where exact analysis is out of reach, and could provide a bridge between the theoretical framework developed here and numerical investigations of large systems.

\section*{Acknowledgements}The work of SB was supported by the European Union’s Horizon 2020 research and innovation programme under the Marie Skłodowska-Curie grant agreement no.\ 101034253 while affiliated with Leiden University. SB was further supported through “Gruppo Nazionale per l’Analisi Matematica, la Probabilità e le loro Applicazioni” (GNAMPA-INdAM).
 \includegraphics[height=1em]{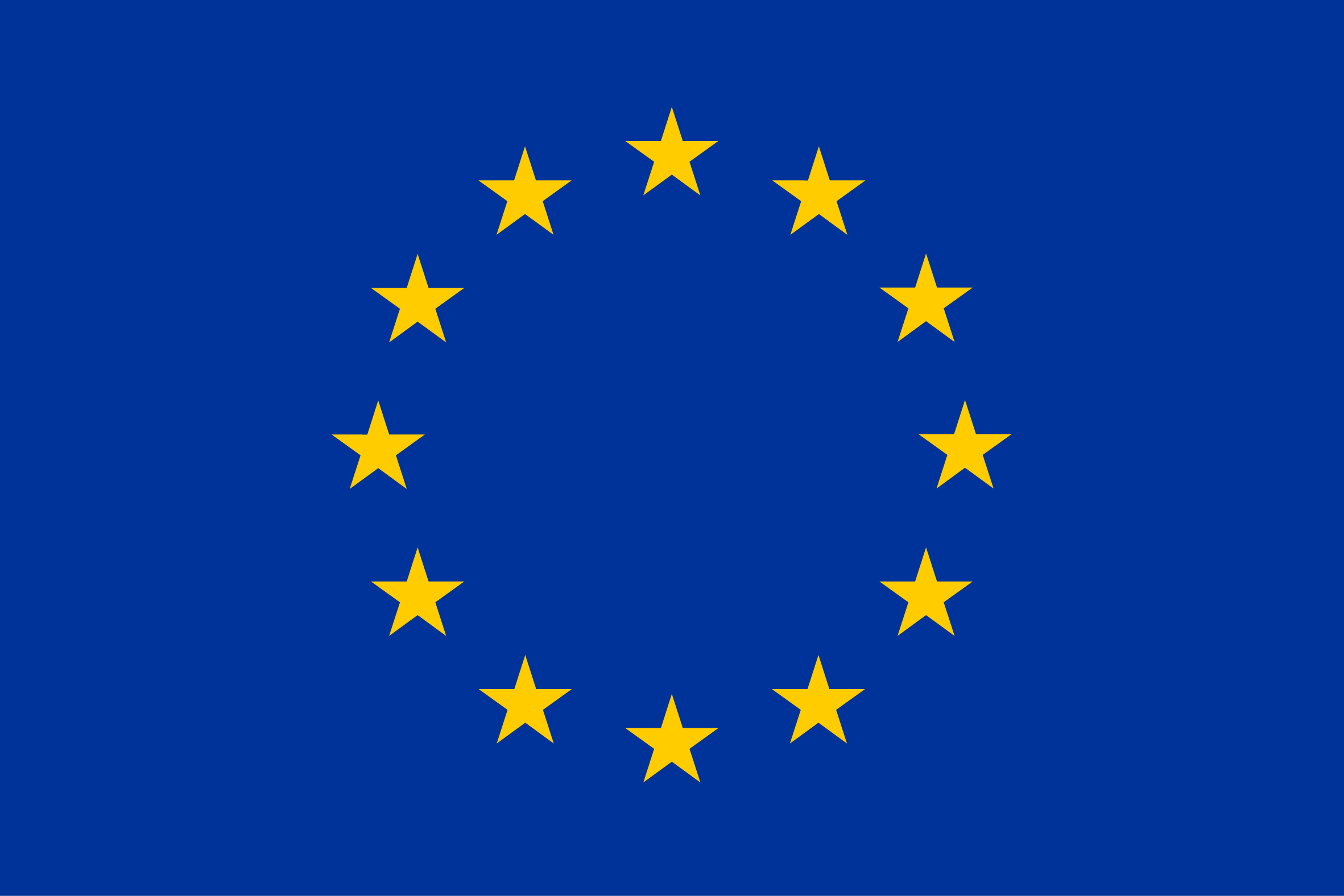}


\small
\bibliographystyle{abbrvnat}
\bibliography{biblio.bib}

\end{document}